\documentclass[12pt,a4paper]{amsart}
\usepackage{amssymb}

     \usepackage{epsfig}
    \usepackage{graphicx}

\newcommand{\eps}{\varepsilon}

\newcommand{\N}{{\mathbb N}}
\newcommand{\C}{{\mathbb C}}
\newcommand{\Z}{{\mathbb Z}}
\newcommand{\R}{{\mathbb R}}
\newcommand{\h}{{\mathbb H}}

\newcommand{\re}{{\rm Re\,}}
\newcommand{\im}{{\rm Im\,}}

\newcommand{\tef}{transcendental entire function}
\newcommand{\tmf}{transcendental meromorphic function}

\newcommand{\B}{\mathcal{B}}
\newcommand{\sing}{\operatorname{sing}}

\newcommand{\Blog}{\B_{\log}}
\newcommand{\Blogn}{\B_{\log}^{\operatorname{n}}}
\renewcommand{\H}{\mathbb{H}}
\newcommand{\T}{\mathcal{T}}
\newcommand{\dist}{{\rm dist \,}}

\theoremstyle{plain}
\newtheorem{theorem}{Theorem}[section]
\newtheorem{corollary}[theorem]{Corollary}

\newtheorem{proposition}[theorem]{Proposition}
\newtheorem{lemma}[theorem]{Lemma}
\theoremstyle{definition}
\newtheorem{definition}{Definition}[section]
\theoremstyle{remark}

\begin{document}


\title{Are Devaney hairs fast escaping?}

\author{Lasse Rempe}
\address{Department for Mathematical Sciences \\
         University of Liverpool \\
         Liverpool \\
         L69 7ZL\\
         UK}
\email{lrempe@liv.ac.uk}

\author{Philip J. Rippon}
\address{The Open University \\
   Department of Mathematics and Statistics \\
   Walton Hall\\
   Milton Keynes MK7 6AA\\
   UK}
\email{p.j.rippon@open.ac.uk}

\author{Gwyneth M. Stallard}
\address{The Open University \\
   Department of Mathematics and Statistics \\
   Walton Hall\\
   Milton Keynes MK7 6AA\\
   UK}
\email{g.m.stallard@open.ac.uk}



\thanks{All three authors were supported by the EU Research Training Network CODY. The first author was supported by EPSRC grant EP/E017886/1.}


\begin{abstract}
Beginning with Devaney, several authors have studied transcendental entire functions for which every point
in the escaping set can be connected to infinity by a curve in the escaping set. Such curves are often
called Devaney hairs. We show that, in many cases, every point in such a curve, apart from possibly a
finite endpoint of the curve, belongs to the fast escaping set. We also give an example of a Devaney hair
which lies in a logarithmic tract of a {\tef} and contains no fast escaping points.
\end{abstract}

\maketitle

\begin{center}
{\it For Bob Devaney on the occasion of his sixtieth birthday.}\\
\end{center}

\section{Introduction}
\setcounter{equation}{0}

Let $f:\C\to \C$ be a {\tef} and denote by $f^{n},\,n \in \N\,$, the $n$th iterate
of~$f$. The {\it Fatou set} $F(f)$ is defined to be the set of points $z \in \C$ such
that $(f^{n})_{n \in \N}$ forms a normal family in some neighborhood of $z$.  The
complement of $F(f)$ is called the {\it Julia set} $J(f)$ of $f$. An introduction to the
properties of these sets can be found in~\cite{wB93}.

This paper concerns the {\it escaping set} of $f$, defined as follows:
\[
 I(f) = \{z:  f^n(z) \to \infty \text{ as } n \to \infty \}.
\]
 For a general {\tef}~$f$, this set was first studied  by
Eremenko~\cite{E} who proved that $I(f)\cap J(f)\ne \emptyset$,
$J(f)=\partial I(f)$ and also that all components of $\overline{I(f)}$ are unbounded. Eremenko then asked whether it is also true that
all components of $I(f)$ are unbounded. Further, he asked whether a stronger statement is true, namely whether every point in $I(f)$ can be joined
to infinity by a curve in $I(f)$. This second question is related to a question of Fatou~\cite{pF26} as to whether there are infinitely many curves to infinity in $I(f) \cap J(f)$.

One family for which the answer to all of these questions is `yes' is the family of functions defined by $f(z)= \lambda e^z$, for $0 < \lambda < 1/e$. For such functions, it was shown by Devaney and Tangerman~\cite{DT86} that $F(f)$ is a completely invariant immediate attracting basin and~$J(f)$ is a so called `Cantor bouquet' consisting of uncountably many disjoint simple curves, each with one finite endpoint and the other endpoint at $\infty$. The set $I(f)$ consists of the open curves
(without endpoints) together with some of their finite endpoints; see~\cite{bK99} and~\cite{lR06}. These open curves are known to be contained in the set of points that `zip to infinity'~\cite{SZ} defined by
\[
Z(f) = \{z\in I(f): \frac{1}{n}\ln\ln |f^n(z)|\to\infty \text{ as
}n\to\infty\}.
\]
It was shown by Rempe~\cite[Theorem~1.4]{lR06} that there are endpoints of such curves that escape to infinity at each possible rate. These curves are often referred to as `Devaney hairs'.

In this paper, we define a {\it Devaney hair} of a {\tmf}~$f$ to be a simple curve $\gamma:[0,\infty) \to J(f)$ such that $\gamma((0,\infty))\subset I(f)$ and
\begin{itemize}
\item[(a)] for each $t>0$, $f^n \to \infty$ as $n\to\infty$ uniformly on $\gamma([t,\infty))$;
\item[(b)] for each $n\in\N$, $f^n(\gamma)$ is a simple curve connecting $f^n(\gamma(0))$ to $\infty$.
\end{itemize}

Note that there are several alternative definitions of a Devaney hair that have been used by other authors. Our definition is closely related to the definition of a ray tail given in~\cite{RRRS} -- a ray tail is a Devaney hair for which $\gamma(0) \in I(f)$ and $f^n \to \infty$ as $n \to \infty$ uniformly on $\gamma([0,\infty))$.

The exponential functions belong to the Eremenko-Lyubich class, defined as
\[
 \B = \{f: f \text{ is a transcendental entire function and }\sing (f^{-1}) \text{ is bounded}\},
\]
where the set $ \sing (f^{-1})$ consists of the critical values and the finite asymptotic values of $f$. It was shown by Eremenko and Lyubich~\cite{EL} that $I(f) \subset J(f)$ for every function $f$ in the class $\B$. Several authors have studied functions in the class $\B$ and it is now known that, for many functions in this class, the escaping set has properties similar to those described above for the exponential family. The most general result known of this type is the following, proved in~\cite{RRRS}.

\begin{theorem}\label{RRRSth}
Let $f$ be a finite composition of functions of finite order in the class $\B$ and let $z_0 \in I(f)$. Then $z_0$ can be connected to $\infty$ by a simple curve $\gamma \subset I(f)$ such that
$f^n \to \infty$ uniformly on $\gamma$.
\end{theorem}

{\it Remarks.}\quad 1. In Section~2 we outline the proof of Theorem~\ref{RRRSth}; from the proof it is clear that the curve $\gamma$ is unique and that, for some $k\in\N$, the curve $f^k(\gamma)$ is a Devaney hair in the sense defined above.

2. The order, $\rho$, of a function $f$ is defined by
\[
\rho(f) =  \lim \sup_{r \to \infty} \frac{\ln \ln M(r,f)}{\ln r},
\]
where $M(r,f) = \max_{|z|=r} |f(z)|$.

3. Note that if~$f$ is a finite composition of functions in the class~${\mathcal B}$, then~$f\in{\mathcal B}$.

Clearly, for any function satisfying the hypotheses of Theorem~\ref{RRRSth}, the answers to both of Eremenko's questions and to Fatou's question is yes. It is shown, however, in~\cite{RRRS} that there exists a function in the class $\B$ for which all the path-connected components of $I(f)$ are points and so, for this function, the answer to the stronger form of Eremenko's question is no. However, the answer to the weaker form of Eremenko's question for this function is yes by a result of Rempe~\cite{lR07} who showed that all the components of $I(f)$ are unbounded whenever $P(f) = \bigcup_{n=0}^{\infty}f^n(\sing (f^{-1}))$ is bounded.

Bara\'nski~\cite{kB07} also studied functions of finite order in the class $\B$ but with the extra condition that $\sing (f^{-1})$ is contained in a compact subset of  the immediate basin of an attracting fixed point of $f$. With this extra condition he was able to prove the stronger result that $J(f)$ consists of a family of disjoint curves with similar properties to the exponential functions studied by Devaney.

Here we show that points on the curves obtained in Theorem~\ref{RRRSth}, with the possible exception of the endpoints, escape to infinity `as fast as possible'. More precisely we show that these points belong to the {\it fast escaping set} defined by
\[
 A(f) = \{z:\, \text{there exists $L \in \N$ such that } |f^{n+L}(z)| \geq M^n(R,f), \text{ for all $n\in\N$}\}.
\]
Here, $M^{n+1}(R,f) = M\left( M^n(R,f),f \right)$ and $R$ is any value with 
 $R > \min_{z \in J(f)}|z|$. Note that $A(f)$ is independent of the value of~$R$.

The set $A(f)$ was introduced by Bergweiler and Hinkkanen in~\cite{BH99} and has many strong properties -- for example, it was shown in~\cite{RS05} that all the components of $A(f)$ are unbounded. Note that $A(f)$ is a subset of the set $Z(f)$ defined earlier; see~\cite{BH99}.

The main result of this paper is the following; for completeness, we include the conclusions of Theorem~\ref{RRRSth} in this result.

\begin{theorem}\label{main}
Let $f$ be a finite composition of functions of finite order in the class $\B$ and let $z_0 \in I(f)$.
Then $z_0$ can be connected to $\infty$ by a simple curve $\gamma \subset I(f)$ such that
\begin{itemize}
\item[(a)] $f^n\to \infty$ as $n\to\infty$ uniformly on $\gamma$,
\item[(b)] there exists $k \in \N$ such that $f^k(\gamma)$ is a Devaney hair of $f$,
\item[(c)] $\gamma \setminus \{z_0\} \subset A(f)$.
\end{itemize}
\end{theorem}

Theorem~\ref{main} states that all points on the curves in Theorem~\ref{RRRSth}, except possibly endpoints, are fast escaping.
We can also use Theorem~\ref{main} to show that, for certain functions outside
the class $\B$, the escaping set meets the Julia set in curves all points of which, except possibly endpoints, are fast escaping.

\begin{theorem}\label{Fatou}
Let $f$ be a {\tef} such that
\[ \pi(f(z))=g(\pi(z)),\quad \text{for }z\in\C,\]
where $\pi(z)=e^{az}$, $a\ne 0$, and $g$ is a {\tef} which is a self
map of $\C^*=\C\setminus\{0\}$.
If $g$ is a finite composition of functions of finite order in the class $\B$, $g$ has an attracting fixed point at~$0$, and
$\sing (g^{-1})$ is contained in a compact subset of the immediate
basin of attraction of~$g$ containing~$0$, then
\begin{itemize}
\item[(a)] $J(f)$ consists of a family of disjoint
curves tending to $\infty$, each homeomorphic to $[0,\infty)$;
 \item[(b)] $I(f)\cap J(f)$
consists of the disjoint curves comprising~$J(f)$, omitting some of
their endpoints; each of these curves, except possibly its  endpoint,
is contained in $A(f)$.
\end{itemize}
\end{theorem}
{\it Remarks.}\quad 1. Any {\tef} which is a self-map of $\C^*$ is of
the form $g(z)=z^m\exp h(z)$, where $m\in \N$ and $h$ is
non-constant entire. Thus the functions covered by Theorem~\ref{Fatou} are all of the form
$f(z)=mz+h(e^{az})$, where $m\in\N$, $a\ne 0$ and $h$ is non-constant entire.

2. Since $\sing (g^{-1})$ is contained in a compact subset of the immediate basin of attraction of~ $g$ containing~$0$, this immediate basin of attraction is the only Fatou component of~$g$; see~\cite[page~396]{BK07}. Therefore, the Fatou set
of~$f$ is a completely invariant Baker domain that contains $\sing(f^{-1})$.

The plan of the paper is as follows. In Section 2 we describe the background to Theorem~\ref{main} and include a summary of the main steps in the proof of Theorem~\ref{RRRSth} since these ideas are used in our proofs. In Section 3 we give the proof of Theorem~\ref{main}. In Section 4 we prove Theorem~\ref{Fatou} and show that it applies to various families of functions, including the function
$f(z)=z+1+e^{-z}$ studied by Fatou in~\cite{pF26}. In Section 5 we show that the conclusions of Theorem~\ref{main} also hold for many functions of {\it infinite} order in the class $\B$, and in Section~6 we give an example of a Devaney hair which lies in a logarithmic tract of a {\tef} of infinite order and contains no fast escaping points.

\section{Background}
\setcounter{equation}{0}

Theorem~\ref{RRRSth} was proved in~\cite{RRRS} by using logarithmic transforms of functions in the class $\B$. Here we introduce the techniques involved in that proof.

Let $f$ be a function in the class $\B$, let $D$ denote a bounded Jordan domain such that $\sing (f^{-1}) \cup \{0,f(0)\} \subset D$ and let $W = \C \setminus \overline{D}$.  It is known that each component $V$ of $\{z: f(z) \in W\}$ is a {\it logarithmic tract}; that is, $V$ is simply connected and $f:V \to W$ is a universal covering. The main tool used by Eremenko and Lyubich in~\cite{EL} to study such functions was a {\it logarithmic change of variable}. Let $H = \exp^{-1}(W)$. The map $\exp: H \to W$ is also a universal covering and so there exists a biholomorphic map $G:V \to H$ such that $f = \exp \circ G$. For each $w \in \exp^{-1}(V)$ we define the function $F$ by $F(w) = G(e^w)$ so that $\exp F(w) = f(e^w)$. Since we can define $F$ on $\exp^{-1}(V)$ for each  component $V$ of $\{z: f(z) \in W\}$, this defines a map $F : \exp^{-1}(\{z: f(z) \in W\}) \to H$. We say that $F$ is a {\it logarithmic transform} of~$f$.

Following the notation used in~\cite{RRRS} we denote the domain of $F$ by $\mathcal{T}$ and refer to the components of $\mathcal{T}$ as the {\it tracts} of $F$. As stated in~\cite{RRRS}, the function $F: \mathcal{T} \to H$ has the following properties.

\begin{enumerate}
\item $H$ is a $2\pi i$-periodic Jordan domain that contains a right half-plane.

\item Every component of $\mathcal{T}$ is an unbounded Jordan domain with real parts bounded below, but unbounded from above.

\item The components of $\mathcal{T}$ have disjoint closures and accumulate only at infinity; that is, if $z_n \in \mathcal{T}$ is a sequence of points all belonging to different tracts, then $z_n \to \infty$.

\item For every component $T$ of $\mathcal{T}$, $F:T \to H$ is a conformal isomorphism -- in particular, $F$ extends continuously to the closure $\overline{\mathcal{T}}$ of $\mathcal{T}$.

\item For every component $T$ of $\mathcal{T}$, $\exp|_{\mathcal{T}}$ is injective.

\item $\mathcal{T}$ is invariant under translation by $2 \pi i$.
\end{enumerate}

The class of functions $F: \mathcal{T} \to H$, for which $F$, $\mathcal{T}$ and $H$ have these properties, is denoted by $\B_{\log}$.

{\it Remarks.} 1. Functions in the class $\B_{\log}$ are sometimes assumed to be periodic -- see, for example,~\cite{lR08}. In this paper (and also in~\cite{RRRS}) this assumption is not made.

2. Functions in the class $\B_{\log}$ do not have to be logarithmic transforms of functions in the class $\B$. For example, suppose that $f$ is any meromorphic function with a logarithmic tract $V$. Then a function $F \in \B_{\log}$ with domain $\exp^{-1}(V)$ can be defined as above. As before, the function~$F$ is said to be a logarithmic transform of~$f$. The iteration of meromorphic functions with a logarithmic tract is discussed in~\cite{BRS08}. In later sections, we include remarks to indicate how our results can be generalised to meromorphic functions with a logarithmic tract $V$. These generalisations concern the sets
\[
 I(f,V) = \{z \in V \cap I(f) : f^n(z) \in V \mbox{ for all } n \in \N\}
\]
and
\begin{eqnarray*}
A'(f,V) & = & \{z \in V: \mbox{ there exists } L \in \N \mbox{ such that } f^n(z) \in V \\
&& \mbox{ and } f^{n+L}(z) \geq M_V^n(R,f), \mbox{ for all } n \in \N\},
\end{eqnarray*}
where $R$ is sufficiently large and $M_V(r,f) = \max_{z \in V, |z| = R}|f(z)|$. (We use the notation $A'(f,V)$ in order to maintain consistency with~\cite{BRS08} where $A(f,V)$ is used to denote the union of the set $A'(f,V)$ together with all its pre-images.)

The following key property of functions in the class $\B_{\log}$ was proved first by Eremenko and Lyubich~\cite{EL} for logarithmic transforms of functions in the class~$\B$. It enabled them to prove that if $f$ is in the class $\B$ then $I(f) \subset J(f)$.

\begin{lemma}\label{expansion}
Let $F : \mathcal{T} \to H$ be a function in the class $\B_{\log}$ with $H \supset \{w: \re w > R\}$. Then, for each $w \in T$ with $\re F(w) > R$,
\[
  |F'(w)| \geq \frac{1}{4 \pi} (\re F(w) -  R).
\]
\end{lemma}

It follows from Lemma~\ref{expansion} that if $F:\T\to H$ is in $\B_{\log}$ then there exists $R_0 > 0$ such that
\begin{equation}\label{exp}
|F'(w)| \geq 2 \mbox{ if } w \in \mathcal{T} \mbox{ and } \re F(w) \geq R_0.
\end{equation}
We say that a function $F: \T \to H$ in the class $\Blog$ is {\it normalised} if $H$ is the right half-plane $\h = \{w: \re w > 0\}$
and~\eqref{exp} holds for all $w \in \mathcal{T}$. We denote the class of all such functions by $\Blogn$.

Note that we can pass from any function $F \in \Blog$ to a normalised one by restricting $F$ to $\mathcal{T}' = F^{-1}(\{w: \re w > R_0\})$, where $R_0$ is as above, and applying the change of variable $\zeta = w - R_0$. For this reason, it is usually no loss of generality to assume that $F \in \Blogn$.

For any function $F\in\B_{\log}$, we define the {\it escaping set} of~$F$ to be
\[
 I(F) = \{w \in \overline{\mathcal{T}} : F^n(w) \in \overline{\mathcal{T}}, \mbox{ for all } n \in \N, \mbox{ and } \re F^n(w) \to \infty \mbox{ as } n \to \infty\}.
\]
This is a subset of the {\it Julia set} of~$F$ defined by
\[
  J(F) = \{w \in \overline{\mathcal{T}}: F^n(w) \in \overline{\mathcal{T}}, \mbox{ for all } n \in \N \}.
\]
We define the {\it fast escaping set} of $F$ by
\begin{eqnarray*}
  A(F)& = & \{w \in \overline{\mathcal{T}}: \mbox{ there exists } L \in \N \mbox{ such that } F^n(w) \in \overline{\mathcal{T}} \\
  &&\mbox{ and } \re F^{n+L}(w) \geq M^n(R,F), \mbox{ for all } n \in \N\},
\end{eqnarray*}
where
\[
 M(r,F) = \max_{\re w = r} \re F(w),
\]
and $R$ is any value so large that $M(r,F)>r$ for $r\ge R$. Note that $M(r,F)$ is an increasing function of $r$ by the maximum principle so, with this definition, $A(F)$ is independent of $R$.

Now suppose that $f\in{\mathcal B}$ and that $F: \exp^{-1}(\{z : f(z) \in W\}) \to H$ is a logarithmic transform of~$f$. If $z \in I(f)$, then there exists $N \in \N$ such that $f^n(z) \in W$, for $n\ge N$, and so $f^N(z) \in \exp I(F)$. Similarly, if $z \in A(f)$, then there exists $N \in \N$ such that $f^N(z) \in \exp A(F)$. On the other hand, if $w$ lies in $I(F)$ or $A(F)$, then $z=\exp w$ lies in $I(f)$ or $A(f)$. Also, if $F \in \Blogn$, then $\exp J(F)\subset J(f)$; this follows from~\eqref{exp} using the same argument as that used by Eremenko and Lyubich in~\cite{EL} to show that $I(f) \subset J(f)$ if $f \in \B$. Similar statements are true if $f$ is a {\tmf} with a logarithmic tract~$V$, and $F\in{\mathcal B}_{\log}$ (respectively $\Blogn$) with domain $\exp^{-1}(V)$ is a logarithmic transform of~$f$.

If $w \in J(F)$ then, for each $n \geq 0$, $F^n(w) \in \overline{T_n}$ for some tract $T_n$. We say that $\underline{s} = T_0T_1T_2\ldots$ is the {\it (external) address} (or {\it itinerary}) of $w$. The {\it shift operator} $\sigma$ acts on an external address as follows: if $\underline{s} = T_0T_1T_2\ldots$, then
$\sigma(\underline{s}) = T_1T_2\ldots$. Note that if a point $w$ has address $\underline{s}$ then $F^k(w)$ has address $\sigma^k(\underline{s})$. We denote the set of escaping points of~$F$ with the same  address by
\[
I_{\underline{s}}(F) = \{w \in I(F): w \mbox{ has address } \underline{s}\}.
\]
Similarly, we define
\[
J_{\underline{s}}(F) = \{w: F^n(w) \in \overline{\mathcal{T}}, \mbox{ for all } n \geq 0, \mbox{ and }
w \mbox{ has address } \underline{s}\}.
\]

We now outline the key steps in the proof of Theorem~\ref{RRRSth}, given in~\cite{RRRS}. A key concept in the proof of Theorem~\ref{RRRSth} is the idea of a linear head-start condition.

\begin{definition}
Let $F \in \B_{\log}$. Then $F$ satisfies a {\it linear head-start condition} with constants $K>1$ and $M>0$
if,
\begin{enumerate}
\item for all tracts $T$, $T'$ of $F$ and all $w, \zeta \in \overline{T}$ with $F(w), F(\zeta) \in \overline{T'}$,
     \[
     \re w > K(\re \zeta)^+ + M \Rightarrow \re F(w) > K (\re F(\zeta))^+ + M;
     \]

\item     for each address $\underline{s}$ and each pair of distinct points $w, \zeta \in J_{\underline{s}}(F)$, there exists $N \in \N$ such that either
    \[\re F^N(w) > K (\re F^N(\zeta))^+ + M \; \mbox{ or } \; \re F^N(\zeta) > K (\re F^N(w))^+ + M.
    \]
\end{enumerate}
\end{definition}
{\it Remark.} We use $t^+$ to denote $\max\{t, 0\}$.

 When $F\in \B_{\log}$ satisfies this linear head-start condition, we can define a simple ordering on points in $J_{\underline{s}}(F)$ -- we say that $w_1 \succ w_2$ if there exists $N \in \N$ such that $\re F^N(w_1) > K (\re F^N(w_2))^+ + M$. By using topological arguments, it is shown in~\cite{RRRS} that all the components of $J_{\underline{s}}(F)$ are curves for which the topological ordering agrees with the ordering given by the linear head-start condition. The following result is proved in~\cite[proof of Theorem 4.2]{RRRS}.

\begin{lemma}\label{RRRSle}
Let $F \in \B_{\log}$ and suppose that $F$ satisfies a linear head-start condition. If $w_0 \in I_{\underline{s}}(F)$, then there exists $k_0 \in \N$ such that, for all $k \geq k_0$, $\Gamma_k = \{w \in I_{\sigma^k(\underline{s})}(F): w \succeq F^k(w_0)\}$ is a simple curve connecting $F^k(w_0)$ to $\infty$. Also, $\re F^n(w) \to \infty$ uniformly on $\Gamma_k$.
\end{lemma}

In~\cite[Proposition 5.4]{RRRS} it is shown that a function $F \in \Blogn$ satisfies a linear head-start condition, for some $K>1$ and $M>0$,  whenever all the tracts of $F$ have uniformly bounded slope and uniformly bounded wiggling. These concepts are defined as follows.

\begin{definition}\label{slope}
Let $F \in \B_{\log}$. A tract $T$ of $F$ has {\it bounded slope} with constants $\alpha, \beta > 0$ if
\[
|\im w_1 - \im w_2| \leq \alpha \max \{\re w_1, \re w_2,0 \} + \beta, \quad \mbox{ for all } w_1, w_2 \in T.
\]
If all the tracts of $F$ have bounded slope for these constants, then we say that the tracts have {\it uniformly bounded slope} and $F \in \B_{\log}(\alpha, \beta)$. If $F$ is normalised then we say that $F \in \Blogn(\alpha, \beta)$.
\end{definition}

\begin{definition}
Let $F \in \B_{\log}$. A tract $T$ of $F$ has {\it bounded wiggling} with constants $K' > 1$ and $\mu > 0$ if, for each point $w_0 \in \overline{T}$, every point~$w$ on the hyperbolic geodesic of $T$ that connects $w_0$ to $\infty$ satisfies
\[
 (\re w)^+ > \frac{1}{K'} \re w_0 - \mu.
\]
If all the tracts of $F$ have bounded wiggling for these constants, then we say that the tracts have {\it uniformly bounded wiggling}.
\end{definition}

It is shown in~\cite[Theorem 5.6]{RRRS} that if $F$ is a function of finite order in the class~$\Blogn$ then the tracts of $F$ have uniformly bounded slope and uniformly bounded wiggling, and so $F$ satisfies a linear head-start condition. It is further shown in~\cite[Lemma 5.7]{RRRS} that any finite composition of functions in $\Blogn$ whose tracts have uniformly bounded slope and uniformly bounded wiggling also satisfies a linear head-start condition.

A function $F \in \B_{\log}$ is said to have {\it finite order} if
\[
 \ln \re F(w) = O(\re w) \quad \mbox{ as } \re w \to \infty \mbox{ in } \mathcal{T}.
\]
So, if $f \in \B$ has finite order and $F$ is a logarithmic transform of $f$, then $F$ also has finite order. Thus, for each function $f$ satisfying the hypotheses of Theorem~\ref{RRRSth}, there is a logarithmic transform $F$ of $f$ satisfying a linear head-start condition.

Now suppose that $f$ is any function in the class $\B$ such that a logarithmic transform $F$ of $f$ satisfies a linear head-start condition. Recall that if $z_0 \in I(f)$ then there exists $N \in \N$ such that $f^{N}(z_0) \in \exp I(F)$. So, by Lemma~\ref{RRRSle}, there exists $k \in \N$ such that $f^{N+k}(z_0)$ is connected to $\infty$ by a curve $\exp \Gamma$ in $I(f)$ on which $f^n(z) \to \infty$ uniformly. Thus $z_0$ is connected to $\infty$ by a curve $\gamma$  in $I(f)$ on which $f^n(z) \to \infty$ uniformly. This proves Theorem~\ref{RRRSth}. Further, it shows that the conclusions of Theorem~\ref{RRRSth} hold whenever $f$ is a function in the class $\B$ for which a logarithmic transform of $f$ satisfies a linear head-start condition.

We end this section with two results about the class $\Blogn$ that will be useful in the proof of Theorem~\ref{main}. Note that both of these results use the fact that, for a function $F \in \Blogn$, if $w$ belongs to a tract of $F$, then $\re F(w) > 0$. They also both include a condition to the effect that there exist $\alpha, \beta > 0$ such that, if $T$ is a tract of $F$ and  $w, \zeta \in \overline {T}$ then
\[
|\im F(w) - \im F(\zeta)| \leq \alpha \max\{\re F(w), \re F(\zeta)\} + \beta.
\]
Note that this condition is automatically satisfied if $F(w)$ and $F(\zeta)$ both lie in a tract  with bounded slope with constants $\alpha, \beta > 0$.

The first result is part of~\cite[Lemma 5.2]{RRRS}.

\begin{lemma}\label{5.2}
Let $F \in \Blogn$ and let $\alpha, \beta > 0$. Let $T$ be a tract of $F$ and suppose that $w, \zeta \in \overline{T}$ satisfy
 \[
  \re F(w) \geq \re F(\zeta) \; \mbox{ and } \; |\im F(w) - \im F(\zeta)| \leq \alpha \re F(w) + \beta.
 \]

 \begin{itemize}
\item[(a)]
   There exists $\delta = \delta(\alpha, \beta)$ so large that, if $|w-\zeta|\geq \delta$, then
   \[
   \re F(w) > \exp( \tfrac{1}{16 \pi}|w-\zeta|)\re F(\zeta).
   \]

\item[(b)]
 If $K\geq 1$, then there exists $\Delta = \Delta(\alpha, \beta,K)$ so large that, if $w, \zeta \in J(F)$ have the same address and $|w - \zeta| \geq \Delta$, then, for all $n \in \N$,
\[
 \re F^n(w) > K \re F^n(\zeta) + |w-\zeta| \; \mbox{ or } \; \re F^n(\zeta) > K \re F^n(w) + |w-\zeta|.
\]
\end{itemize}
\end{lemma}

The next result follows from~\cite[Proposition 5.4, proof that (b) implies (c)]{RRRS}.

\begin{lemma}\label{5.4}
Let $F \in \Blogn$, let $T$ be a tract of $F$ with bounded wiggling, with constants $K'>1$ and $\mu > 0$, let $K >1$, let $\alpha, \beta > 0$ and let $\delta(\alpha, \beta)$ be the constant given in Lemma~\ref{5.2}. There exists a constant $M > \delta(\alpha,\beta)$ depending only on the constants $K', \mu , \alpha, \beta, K$ such that, if $w, \zeta \in \overline{T}$ with
\[
\re w > K(\re \zeta)^+ + M
 \]
 and
 \[ |\im F(w) - \im F(\zeta)| \leq \alpha \max\{\re F(w), \re F(\zeta)\} + \beta,
\]
then
\[
 \re F(w) > K \re F(\zeta) + M.
\]
\end{lemma}

\section{Proof of Theorem~\ref{main}}
\setcounter{equation}{0}

We deduce Theorem~\ref{main} from a result concerning functions in the class $\Blogn$. Here we define a {\it Devaney hair} of a function $F\in \B_{\log}$ to be a simple curve $\Gamma:[0,\infty) \to J(F)$ such that $\Gamma((0,\infty))\subset I(F)$ and
\begin{itemize}
\item[(a)] for each $t>0$, $\re F^n \to \infty$ as $n\to\infty$ uniformly on $\Gamma([t,\infty))$;
\item[(b)] for each $n\in\N$, $F^n(\Gamma)$ is a simple curve connecting $F^n(\Gamma(0))$ to $\infty$.
\end{itemize}

Parts (a) and (b) of the following theorem follow from the results in \cite{RRRS} that we described in Section 2; we include them here for completeness.

\begin{theorem}\label{mainlog}
Let $F$ be a finite composition of functions of finite order in the class $\Blogn$ and let $w_0' \in I(F)$. Then there exists $k \in \N$ such that $w_0 = F^k(w_0')$ can be connected to $\infty$ by a curve $\Gamma \subset I(F)$ such that
\begin{itemize}
\item[(a)] $\re F^n\to \infty$ as $n\to\infty$ uniformly on $\Gamma$,
\item[(b)] $\Gamma$ is a Devaney hair of $F$,
\item[(c)] $\Gamma \setminus \{w_0\} \subset A(F)$.
\end{itemize}
\end{theorem}

We begin by proving the following lemma.

\begin{lemma}\label{fast}
Let $F \in \Blogn$, let $K>1$ and let $\alpha, \beta > 0$. Let $T$ be a tract of $F$ with bounded wiggling, with constants $K'>1$ and $\mu > 0$, and let $M>0$ be the constant given by Lemma~\ref{5.4}. There exists a constant $\eps>0$, depending only on $K$, such that, if $w, \zeta \in \overline{T}$,
\begin{equation}\label{wzeta}
\re w>K(\re \zeta)^+ + M
\end{equation}
and
\begin{equation}\label{Fwzeta}
|\im F(w) - \im F(\zeta)| \leq \alpha \max\{ \re F(w), \re F(\zeta)\} + \beta,
\end{equation}
then
\[
\re F(w)>\exp(\eps\re w)\re F(\zeta).
\]
\end{lemma}

\begin{proof}
    By~\eqref{wzeta}, we have
 \begin{equation}\label{distance}
|w-\zeta|\ge \re w - \re \zeta \ge (1-1/K)\re w.
 \end{equation}
 By Lemma~\ref{5.4}, we have $\re F(w)>\re F(\zeta)$. So, since $M > \delta(\alpha, \beta)$ and, by (3.1), $|w-\zeta|>M$,  it follows from Lemma~\ref{5.2}(a) and~\eqref{distance} that
 \begin{equation}\label{difference}
 \re F(w) >
  \exp(\tfrac{1}{16\pi}|w - \zeta|)\re F(\zeta)
 \geq  \exp(\tfrac{1}{16\pi}(1-1/K)\re w)\re F(\zeta).
\end{equation}
The result now follows on taking $\eps = \tfrac{1}{16\pi}(1-1/K)$.
\end{proof}

The next result follows from Lemma~\ref{fast}.

\begin{corollary}\label{cor}
Let $F$ be a finite composition of functions $F_i$, $1 \leq i \leq p$, in the class $\Blogn$ such that the tracts of $F_i$, $1\le i\le p$, have uniformly bounded slope and uniformly bounded wiggling. Let $w_0' \in I(F)$. Then there exist $\eps > 0$ and $k,N_0 \in \N$ such that $w_0 = F^k(w_0')$ can be connected to $\infty$ by a curve $\Gamma \subset I(F)$ and, for all  $w_1 \in \Gamma \setminus  \{w_0\}$,
\[\re F^{n+m}(w_1)>\omega^{pn}(\re F^{m}(w_0)), \; \mbox{ for } n \in \N, \; m \geq N_0,\]
where $\omega(t)=\exp(\eps t)$.
\end{corollary}

\begin{proof}
 Let $w_0'\in I(F)$ and let $\underline{s}$ denote the address of $w'_0$ with respect to the function~$F$. Since the tracts of~$F_i$, $1\le i \le p$, have uniformly bounded slope and uniformly bounded wiggling, $F$ satisfies a linear head-start condition, as noted in Section~2. By Lemma~\ref{RRRSle}, for all large $k$ the set $\Gamma = \{w \in I_{\sigma^k(\underline{s})}(F): w \succeq F^k(w'_0)\}$ is a simple curve connecting $F^k(w'_0)$ to $\infty$ within a single tract of~$F$, on which $\re F^n(w)\to\infty$ uniformly. We assume $k$ to be so large that
\[
F^n(\Gamma)\subset\{w:\re w>C\},\quad\text{for }n\ge k,
\]
where
\[
C=\max\{\re F_p\circ \cdots \circ F_i(w):\re w=0,1\le i\le p\}.
\]
It follows that, for each $n\ge k$ and $1\le i\le p$, the curve
\begin{equation}\label{curves}
F_i\circ \cdots F_1\circ F^n(\Gamma)
\end{equation}
lies entirely in a single tract of the function~$F_{i+1}$. (Here, we take $F_{p+1}=F_1$.)

We now let $w_0 = F^k(w'_0)$ and let $w_1 \in \Gamma \setminus\{w_0\}$ so that $w_1 \succ w_0$. Thus $w_0, w_1 \in I(F)$. By hypothesis, there exist $\alpha, \beta > 0$, $K'>1$ and $\mu > 0$ such that the tracts of $F_i$, $1 \leq i \leq p$, have bounded slope, with constants $\alpha, \beta$, and bounded wiggling, with constants $K', \mu$. Further, the tracts of $F$ have bounded slope with constants $\alpha, \beta$.

Since $w_1 \succ w_0$ with respect to a linear head-start condition for $F$, there exist $K>1$, $M_1>0$, $N \in \N$ such that
\begin{equation}\label{hs}
\re F^n(w_1) > K \re F^n(w_0) + M_1, \mbox{ for } n \geq N.
\end{equation}

Now let $M$ be the constant given by Lemma~\ref{5.4} and $\Delta$ be the constant given by Lemma~\ref{5.2}(b). It follows from~\eqref{hs} and the fact that $w_0 \in I(F)$ that there exists $N_0 \geq N+1$ such that $\re F^{N_0 - 1}(w_1) \geq \re F^{N_0 - 1}(w_0)$ and also $|F^{N_0 - 1}(w_1) - F^{N_0 - 1}(w_0)| \geq \max\{\Delta, M\}$. Since $F \in \Blogn(\alpha, \beta)$, it follows from Lemma~\ref{5.2}(b) that
\begin{equation}\label{N_0}
\re F^{N_0}(w_1) > K \re F^{N_0}(w_0) + M.
\end{equation}
Further, since $w_0 \in I(F)$, we may assume that $N_0$ was chosen sufficiently large to ensure that
\begin{equation}\label{one}
\re F^m(w_0) > 1, \; \mbox{ for } m \geq N_0.
\end{equation}
Then, since $F_i \in \Blogn(\alpha, \beta)$, for $1 \leq i \leq p$, it follows from~\eqref{curves} that we can start from~\eqref{N_0} and apply Lemma~\ref{5.4} repeatedly to deduce that
\begin{equation}\label{est1}
\re F_{i}\circ \cdots \circ F_1 \circ F^{n+m}(w_1) >K \re F_{i}\circ \cdots \circ F_1 \circ F^{n+m}(w_0)+M,
\end{equation}
for $n\in\N$, $m \geq N_0$ and $1 \leq i \leq p$. Finally, we  apply Lemma~\ref{fast} repeatedly with
\[
w=F_{i}\circ \cdots \circ F_1 \circ F^{n+m}(w_1)\quad\text{and}\quad \zeta =F_{i}\circ \cdots \circ F_1 \circ F^{n+m}(w_0),
\]
for $n\in \N$ and $1\le i\le p$. In view of~\eqref{one}, this shows that there exists $\eps > 0$ such that, if $n \in \N$ and $m \geq N_0$, then
\[\re F^{n+m}(w_1)>\omega^{pn}(\re F^{m}(w_0)),\]
where $\omega(t)=\exp(\eps t)$.
\end{proof}

 We now show how Theorem~\ref{mainlog} follows from Corollary~\ref{cor}.

 \begin{proof}[Proof of Theorem~\ref{mainlog}]

 Let $F$ be a finite composition of functions $F_i$, $1 \leq i \leq p$, of finite order in the class $\Blogn$. Then, as noted in Section 2, the tracts of $F_i$ have uniformly bounded slope and uniformly bounded wiggling, and so we can apply Corollary~\ref{cor}. Now let $w_1 \in \Gamma \setminus \{w_0\}$, where $\Gamma$ and $w_0$ are as described in Corollary~\ref{cor}. In order to prove Theorem~\ref{mainlog}, it is sufficient to show that $w_1 \in A(F)$.

We begin by noting that, since each function $F_i$, $1 \leq i \leq p$, has finite order, there exist $\delta>0$ and $r_0>0$, such that
\[
r< M(r,F_i) \leq \exp(\delta r), \mbox{ for } r \geq r_0, \; 1 \leq i \leq p.
\]
So, for $n \in \N$,
\begin{equation}\label{Omega}
M^n(r,F) \leq \Omega^{np}(r), \mbox{ for } r \geq r_0,
\end{equation}
where $\Omega(r) = \exp(\delta r)$.

The next lemma compares the iterative behaviours of the functions $\omega$ and $\Omega$.

\begin{lemma}\label{omegas}
Let $\omega(r) = \exp(\eps r)$ and $\Omega(r) = \exp(\delta r)$, where $\delta>\eps > 0$. Then, for each $r>0$, there exists $R>0$ such that
\[
\omega^n(R)\ge\Omega^n(r),\quad\text{for }n\in\N.
\]
\end{lemma}
\begin{proof}
 Put $s=2\delta/\eps$ and take $R\ge\max\{sr, \tfrac{2}{\eps} \ln s\}$ so large that $\omega(t)\ge t$, for $t\ge R$. Then, since $s>1$, it is sufficient to show that
\begin{equation}\label{omegaineq}
\Omega^n(r)\le \tfrac{1}{s}\omega^n(R),\quad\text{for } n \in \N.
\end{equation}
      The inequality~\eqref{omegaineq} follows by induction: it clearly holds for $n=1$ and if it holds for $n=k$, then
\[
\Omega^{k+1}(r)\le \Omega(\tfrac{1}{s}\omega^k(R)) < \exp\left(\eps \omega^k(R) + \ln \tfrac{1}{s}\right) = \tfrac{1}{s} \omega^{k+1}(R),
\]
as required.
\end{proof}
Now, choose $r$ so large that $r>\min_{w\in J(F)}\re w$ and $r\ge r_0$, and let $R$ be given by Lemma~\ref{omegas} and $N_0$ be given by Corollary~\ref{cor}. Then take $m \geq N_0$ such that $\re f^{m}(w_0) \ge R$. It follows from Corollary~\ref{cor}, Lemma~\ref{omegas} and~\eqref{Omega} that, for~$n\in\N$,
\[
\re F^{n+m}(w_1) > \omega^{pn}(R)\ge \Omega^{pn}(r) \geq M^{n}(r,F).
\]
Thus $w_1 \in A(F)$, as required.
\end{proof}

\begin{proof}[Proof of Theorem~\ref{main}]
Theorem~\ref{main} follows from Theorem~\ref{mainlog} since the logarithmic transform of a function of finite order in the class $\B$ is a function of finite order in the class $\B_{\log}$. Further, if $z_0 \in I(f)$, and $F$ is a logarithmic transform of $f$, then there exists $N \in \N$ such that $f^N(z_0) = \exp(w_0')$, for some $w_0' \in I(F)$. Thus $z_0$ can be connected to $\infty$ by a simple curve $\gamma \subset I(f)$ which is a pullback under $f^{N+k}$ of $\exp \Gamma$, where $\Gamma$ is as in  Theorem~\ref{mainlog}. The result now follows from the fact that $\exp A(F) \subset A(f)$ and from the complete invariance of $A(f)$.
\end{proof}

{\it Remark.}
 Similarly, Theorem~\ref{mainlog} can be used to show that, if $f$ is an entire function of finite order with a logarithmic tract $V$ and $z_0 \in I(f,V)$, then $z_0$ can be connected to $\infty$ by a curve $\gamma \subset I(f,V)$ such that $\gamma \setminus \{z_0\} \subset A(f)$.

\section{Proof of Theorem~\ref{Fatou} and examples}
\setcounter{equation}{0}
We begin this section by proving Theorem~\ref{Fatou}.

\begin{proof}
It follows from~\cite[Theorem 5.10]{RRRS} that each component of
$J(g)$ is a curve to $\infty$ homeomorphic to $[0,\infty)$, all points of which
lie in $I(g)$ except possibly the endpoint. (Note that this result was first proved by Bara\'nski~\cite[Theorem~C]{kB07} for the case that $g$ has finite order.)

By a result of Bergweiler~\cite{wB95}, we have
\[\pi^{-1}(J(g))=J(f).\]
Since $0\in F(g)$, the set $J(f)$ also consists of a
family of disjoint curves homeomorphic to $[0,\infty)$, each
component $\gamma$ of $J(g)$ corresponding to a countable family of
disjoint congruent curves in $J(f)$ homeomorphic to $[0,\infty)$, say
$\gamma_n$, $n\in \Z$. This proves part~(a).

Clearly each point of $I(f)\cap J(f)$ lies in some component
$\gamma_n$ of $J(f)$ corresponding to a component $\gamma$ of $J(g)$.
On the other hand, for such a component $\gamma$ of $J(g)$ all
points, except possibly the endpoint, lie in $I(g)$ by~\cite[Theorem 5.10]{RRRS} and hence in
$A(g)$, by Theorem~\ref{main}. Bergweiler and Hinkkanen~\cite[Theorem~5]{BH99}
proved that
\[\pi^{-1}(A(g))\subset A(f),\]
so all points, except possibly the endpoints, of the corresponding curves
$\gamma_n$, $n\in \Z$, lie in $A(f)$ and hence lie in $I(f)\cap
J(f)$. Since some of the endpoints of these components of $J(f)$ do
not lie in $I(f)$, because $J(f)$ must include the repelling periodic
cycles of $f$, the proof of part~(b) is complete.
\end{proof}

Theorem~\ref{Fatou} can be applied to all functions of the form
\[f(z)=z+\lambda+e^{-z},\quad \text{where } \re \lambda>0,\]
by taking $\pi(z)=e^{-z}$ and $g(w)=e^{-\lambda}we^{-w}$. Then $g$ is
a {\tef} of order~1 which is a self map of $\C^*$ and it has an
attracting fixed point at~$0$. In this case $\sing (g^{-1})$ consists
of the asymptotic value~$0$ and the critical value $e^{-1-\lambda}$.
To see that this critical value must lie in the immediate basin
of~$0$, we observe that if it does not, then the branch of $g^{-1}$
which maps~$0$ to~$0$ can be analytically continued to the whole of
the basin, as in~\cite[proof of Theorem~2.2]{CG}, which is
impossible.

Theorem~\ref{Fatou} can also be applied to some functions of the form
\[f(z)=mz+\lambda+e^{-z},\quad\text{where }m\in \N, m\ge 2, \lambda\in\C,\]
by taking $\pi(z)=e^{-z}$ and $g(w)=e^{-\lambda}w^m e^{-w}$. Then $g$
is a {\tef} of order~1 which is a self map of $\C^*$ and it has a
super-attracting fixed point at~$0$. In this case $\sing (g^{-1})$
consists of the critical and asymptotic value~$0$, and the critical
value $m^m e^{-m-\lambda}$. The values of $\lambda$ for which this
latter critical value is contained in the immediate basin of~$0$
depend on~$m$. For example, if $\lambda=0$, then it is easy to check
graphically that $m^m e^{-m}$ lies in the immediate basin of~$0$ for
$m=2$ and $m=3$, but it lies in a different component of~$F(g)$ for
$m=4$. The set of $\lambda$ for which Theorem~3 can be applied in
this case includes the half-plane
\[\{\lambda:\re \lambda>1+m(\ln m-1)\},\]
since for these values of $\lambda$ we have $m^m e^{-m-\lambda}\in
\{w:|w|<1/e\}\subset F(g)$, and includes the half-line
\[\{\lambda\in\R:\lambda>(m-1)(\ln (m-1)-1)\},\]
since for these values of $\lambda$ we have $0<g(u)<u$, for $u>0$.

\section{Fast escaping curves for functions of infinite order}
\setcounter{equation}{0}
Let $f$ be a function in the class $\B$ and let $F$ denote a logarithmic transform of $f$ in the class $\Blogn$. As described in Section 2, it was shown in~\cite{RRRS} that the conclusion of Theorem~\ref{RRRSth} holds whenever $F$ satisfies a linear head-start condition and that such a condition is satisfied whenever the tracts of $F$ have uniformly bounded slope and uniformly bounded wiggling. In particular, $F$ satisfies such a condition if $f$ has finite order.

There are, however, many functions of {\it infinite order} in the class $\B$ for which the tracts of $F$ have uniformly bounded slope and uniformly bounded wiggling; for example, the functions studied in~\cite{Sta91} and~\cite{Sta00} are of this type. It is natural to ask whether the conclusion of Theorem~\ref{main} also holds for such functions. In this section we show that this is the case provided that the tracts of $F$ satisfy a further geometric condition (which holds for the functions studied in~\cite{Sta91} and~\cite{Sta00}). In the next section we give an example which shows that it is necessary to impose an extra condition of this type.

We first introduce some notation. Let $F \in \B_{\log}$ and let $T$ be a tract of $F$. If $\zeta \in \overline{T}$ and $a > \re \zeta$, then $L_{\zeta,a}$ denotes the unique component of $\{w: \re w = a\} \cap T$ that separates $\zeta$ from $\infty$ (in~$T$) and that can be joined to $\zeta$ by a path lying, apart from its endpoints, in $\{w: \re w < a\} \cap T$. Note that $L_{\zeta,a}$ is a cross cut of $T$ and $\zeta$ lies in the closure of the bounded component of $T\setminus L_{\zeta,a}$.

We now introduce the notion of a gulf of a tract of $F$ -- the terminology is motivated by the definition of a gulf of a tract of $f$ given in~\cite{EL}.

\begin{definition}
Let $F \in \Blogn$, let $T$ be a tract of $F$ and let $p$ denote the point in $\partial T$ for which $F(p) = 0$.
\begin{itemize}
\item[(a)]
Let $\zeta \in T$. We say that $\zeta$ belongs to a {\it gulf} of $T$ if there exists $a > \max\{\re \zeta, \re p\}$ such that $L_{\zeta,a}$ does not separate~$p$ from~$\infty$.
\item[(b)]
The tract $T$ has {\it bounded gulfs} with constant $C>1$ if
\[
 L_{\zeta,a}\;\text{ separates } p \text{ from }\infty,
\]
for all $\zeta \in T, a>0$, with $\re \zeta \geq \max\{\re p,1\}$ and $a \geq C \re \zeta$.
\end{itemize}
If all the tracts of $F$ have bounded gulfs for the constant $C$, then we say that the tracts have {\it uniformly bounded gulfs}.
\end{definition}

Parts (a) and (b) of the following theorem follow from the results in \cite{RRRS} that we described in Section 2; we include them here for completeness.

\begin{theorem}\label{infinite}
Let $F \in \Blogn$. Suppose that all the tracts of $F$ are translates of each
other by $2n\pi i$, $n \in \Z$, and have bounded slope, bounded wiggling and bounded gulfs. If $w_0' \in
I(F)$, then there exists $k \in \N$ such that $w_0 = F^k(w_0')$ can be connected to $\infty$ by a curve $\Gamma \subset I(F)$ such that
\begin{itemize}
\item[(a)] $\re F^n\to \infty$ as $n\to\infty$ uniformly on $\Gamma$,
\item[(b)] $\Gamma$ is a Devaney hair of $F$,
\item[(c)] $\Gamma \setminus \{w_0\} \subset A(F)$.
\end{itemize}
\end{theorem}
{\it Remarks.}\;1. The example in the next section shows that Theorem~\ref{infinite} is not true if the tracts do not have bounded gulfs.

2. It follows from Theorem~\ref{infinite} that, if $f$ is a meromorphic function with a logarithmic tract $V$, such that a logarithmic transform $F \in \Blogn$ of~$f$ has tracts of bounded slope, bounded wiggling and bounded gulfs, and $z_0 \in I(f,V)$, then~$z_0$ can be connected to $\infty$ by a curve $\gamma \subset I(f,V)$ such that $\gamma \setminus \{z_0\} \subset A'(f,V)$.

3. The result in Remark 2 implies that if $f \in \B$, $f$ has exactly one tract and $z_0 \in I(f)$, then $z_0$ can be connected to $\infty$ by a curve $\gamma \subset I(f)$ such that $\gamma \setminus \{z_0\} \subset A(f)$. If $f$ has more than one tract, then an analogous result holds but the conclusion has to be phrased in terms of points that tend to infinity as fast as possible with respect to a particular sequence of tracts of $f$. Note that such points may not belong to $A(f)$; see~\cite[end of Section 4]{BRS08}. It is also possible to obtain results concerning points which tend to infinity as fast as possible with respect to a given address -- we hope to return to this idea in future work.

In order to prove Theorem~\ref{infinite}, we begin by proving two preliminary results. The first concerns functions with bounded gulfs and bounded wiggling.

\begin{lemma}\label{gulfs}
Let $F \in \Blogn$, let $T$ be a tract of $F$ and let $p$ denote the point in $\partial T$ for which $F(p) = 0$. Suppose that $T$ has bounded gulfs with constant $C>1$ and bounded wiggling with constants $K'>1$ and $\mu>0$. There is a constant $D>1$ depending on $C,K'$ and $\mu$ such that if $A > \max\{\re p,1\}$ and $a\geq DA$, then
\begin{itemize}
\item[(a)] $L_{\zeta,a}=L_{p,a}$ whenever $\zeta\in T$ and $\re \zeta=A$,
\item[(b)]
\(
 \max_{w \in L_{p,a}} \re F(w) \geq \max_{w \in T, \re w = A} \re F(w).
\)
\end{itemize}

\end{lemma}
\begin{proof}
Suppose that $a>2K'\max\{C,\mu\}A$ and put $a'=a/(2K')$. To prove part~(a), we assume (for a contradiction) that $\zeta\in T$ and $\re \zeta =A$, and $L_{\zeta,a}\ne L_{p,a}$. Since~$T$ has bounded gulfs with constant~$C$,
\begin{equation}\label{Lzeta0}
L_{\zeta,a'}\;\text{ separates } p \text{ from }\infty
\end{equation}
and hence
\begin{equation}\label{Lzeta1}
L_{\zeta,a}\;\text{ separates } p \text{ from }\infty.
\end{equation}

Let $T_{p,a}$ denote the bounded component of $T \setminus L_{p,a}$ and $U_{p,a}$ denote its unbounded component. Since $L_{\zeta,a} \neq L_{p,a}$, the cross cut $L_{\zeta,a}$ must belong to either $T_{p,a}$ or $U_{p,a}$.

   If $L_{\zeta,a}\subset T_{p,a}$, then by~\eqref{Lzeta1} any curve from $p$ to $L_{p,a}$ must cross $L_{\zeta,a}$, contrary to the definition of $L_{p,a}$. Thus
\[
 L_{\zeta,a}\subset U_{p,a}.
\]
This implies that
\begin{equation}\label{Lzeta4}
\zeta\in U_{p,a}
 \end{equation}
 for otherwise any curve from $\zeta$ to $L_{\zeta,a}$ must cross $L_{p,a}$, contrary to the definition of $L_{\zeta,a}$. It follows from~\eqref{Lzeta4} that
\begin{equation}\label{Lzeta3}
L_{\zeta,a'}\subset U_{p,a},
\end{equation}
for otherwise any curve from $L_{\zeta,a'}$ to $\zeta$ must cross $L_{p,a}$ and so contain points $w$ such that $\re w = a > a'$, contrary to the definition of $L_{\zeta,a}$.

It follows from~\eqref{Lzeta0} and~\eqref{Lzeta3} that any geodesic from $L_{p,a}$ to $\infty$ passes through $L_{\zeta,a'}$, so by bounded wiggling we have
\[
a'>\frac{a}{K'}-\mu;\quad\text{that is,}\quad\frac{a}{2K'}<\mu,
\]
which contradicts our choice of $a$. Hence $L_{p,a}= L_{\zeta,a}$.

To prove part~(b), let $\zeta$ denote a point in $\{w: \re w = A\} \cap T$ for which $\re F(\zeta) = \max_{w \in T, \re w = A} \re F(w)$. Once again let $T_{p,a}$ denote the bounded component of $T \setminus L_{p,a}$. Note that the boundary of $T_{p,a}$ consists of $L_{p,a}$ together with part of the boundary of $T$. Since $\re F(w) = 0$ for $w \in \partial T$, and $\re F(w) > 0$ for $w \in T$, it follows that
\begin{equation}\label{boundary}
\max_{w \in \partial T_{p,a}} \re F(w) = \max_{w \in L_{p,a}} \re F(w).
\end{equation}
By the maximum principle,
\begin{equation}\label{maxp}
\max_{w \in \partial T_{p,a}} \re F(w) = \max_{w \in \overline{T_{p,a}}} \re F(w).
\end{equation}
By part~(a) and the definition of $L_{\zeta,a}$ we have $\zeta \in T_{p,a}$ and so, by \eqref{boundary} and \eqref{maxp},
\[
 \max_{w \in L_{p,a}} \re F(w) \geq \re F(\zeta) = \max_{w \in T, \re w = A} \re F(w)
\]
as required.
\end{proof}

The second preliminary result is based on Ahlfors' distortion theorem.

\begin{lemma}\label{Ahlfors}
Let $F \in \Blogn$, let $T$ be a tract of $F$ and let $p$ denote the point in $\partial T$ for which $F(p) = 0$. If $a> \re p$ and $b > a + 4 \pi$ then
\begin{equation}\label{2}
\frac{|F(w_b)|}{|F(w_a)|} \geq \exp(\tfrac{1}{2}(b-a) - 4\pi), \mbox{ for } w_b \in L_{p,b}, w_a \in L_{p,a}.
\end{equation}
\end{lemma}
\begin{proof}
For $t> \re p$, let $l(t)$ denote the length of $L_{p,t}$. Then $0 < l(t) \leq 2 \pi$. Now let $g(w) = \log F(w)$. Since $g$ maps $T$ univalently onto the strip $S = \{\xi : |\im \xi| < \pi / 2\}$, we can apply Ahlfors' distortion theorem to $g$. The version of this result in~\cite[page 97]{Nev} states that, if
\[
b > a > \re p\quad \mbox{and}\quad \int_a^b \frac{dt}{l(t)} > 2,
\]
then, for $w_b \in L_{p,b}$ and $w_a \in L_{p,a}$,
\[
\re g(w_b) - \re g(w_a)
 \geq  \pi \int_a^b \frac{dt}{l(t)} - 4 \pi > \frac{1}{2}(b-a) - 4\pi.
\]
and so, by the definition of $g$,
\[
\ln |F(w_b)| - \ln |F(w_a)| > \frac{1}{2}(b-a) - 4\pi.
\]
Since $0 < l(t) \leq 2\pi$, we have
\[
 \int_a^b \frac{dt}{l(t)} \geq \frac{b-a}{2 \pi} > 2 \quad\mbox{if } b-a > 4 \pi,
\]
and so the result holds if $b > a + 4 \pi$.
\end{proof}

We are now in a position to prove Theorem~\ref{infinite}.

\begin{proof}[Proof of Theorem~\ref{infinite}]
 Let $F$ be a function satisfying the conditions of Theorem~\ref{infinite}, let $w_0' \in I(F)$ and let $\underline{s}$ denote the address of $w'_0$. Since the tracts of~$F$ have bounded slope and bounded wiggling, $F$ satisfies a linear head-start condition, as noted in Section~2. Hence, by Lemma~\ref{RRRSle}, there exists $k \in \N$ such that $\Gamma = \{w \in I_{\sigma^k(\underline{s})}(F): w \succeq F^k(w'_0)\}$ is a simple curve in $I(F)$ connecting $F^k(w'_0)$ to $\infty$. We let $w_0 = F^k(w'_0)$ and let $w_1$ denote a point in $\Gamma$ with $w_1 \succ w_0$. In order to prove Theorem~\ref{infinite}, it is sufficient to show that $w_1 \in A(F)$.

 The following result is the main step in the proof of Theorem~\ref{infinite}.

\begin{lemma}\label{A(f)}
Let $w_1$ be as described above. There exist $N_1 \in \N$ and $ \eps \in (0, 1)$ such that, for all $n \geq N_1$,
\[
\re  F^{n+1}(w_1) > \frac{1}{\eps}M(\eps \re F^n(w_1),F).
\]
\end{lemma}
\begin{proof}
Recall that the tracts of $F$ have bounded slope with constants $\alpha, \beta > 0$, bounded wiggling with constants $K'>1$ and $\mu >0$, and bounded gulfs with constant $C>1$. Also,
the tracts of $F$ are translates of each other by $2n\pi i$, $n \in \Z$, and so, for each tract $T$,
\[
 \max_{w \in T, \re w = r} \re F(w) = M(r,F).
\]
 Let $T_n$ denote the tract of $F$ containing $F^n(w_1)$ and let $p_n$ denote the point in $\partial T_n$ for which $F(p_n) = 0$. Note that all the points $p_n$ are vertical translates of each other. We begin by applying Lemma~\ref{Ahlfors} to the tract $T_n$ with
\[
a_n = \frac{1}{4K'}\re F^n(w_1) \; \mbox{ and } \;
 b_n = \frac{1}{2K'} \re  F^n(w_1).
  \]

Since $w_1 \in I(F)$, the hypotheses of Lemma~\ref{Ahlfors} are satisfied provided that $n$ is sufficiently large. We now take $w_{b_n} \in L_{p_n,b_n}$ and $w_{a_n} \in L_{p_n,a_n}$ to be points on the unbounded curve $F^n(\Gamma)$. To see that this is possible, note that $F^n(w_0)$ is also on the curve $F^n(\Gamma)$ and, by Lemma~\ref{gulfs},  $L_{F^n(w_0),a_n} = L_{p_n,a_n}$ for large $n \in \N$. (The hypotheses of Lemma~\ref{gulfs} are satisfied for large $n \in \N$ since $w_0 \in I(F)$, the tracts of $F$ have bounded slope and $w_1 \succ w_0$ with respect to a linear head-start condition for $F$ so that, for $n$ sufficiently large, we can apply Lemma~\ref{5.2}(a) to $F^n(w_1)$ and $F^n(w_0)$ to deduce that
 \[
 \frac{a_n}{\re F^n(w_0)} = \frac{\re F^n(w_1)}{4K' \re F^n(w_0)} \to \infty \mbox{ as } n \to \infty.)
 \]
  Since $w_0 \in I(F)$ and $w_{a_n}, w_{b_n} \succ F^n(w_0)$, it follows from Lemma~\ref{Ahlfors} that, if $n$ is sufficiently large, then
\begin{equation}\label{3}
\frac{|F(w_{b_n})|}{|F(w_{a_n})|} \geq \exp(\tfrac{1}{2}(b_n-a_n) - 4\pi).
\end{equation}

Now take $c=(4(\alpha+2))^{-1}$. It follows from~\eqref{3} and the fact that
\[b_n - a_n = \frac{1}{4K'}\re F^n(w_1) \to \infty\;\text{ as } n \to \infty\]
that, for large $n$,
\begin{equation}\label{Lestimate}
|F(w_{b_n})|>(1/c)|F(w_{a_n})|.
\end{equation}
Since all points on the curve $F^n(\Gamma)$ have the same address, $F(w_{b_n})$ and
$F(w_{a_n})$ both belong to the same tract of $F$, which has
 bounded slope with constants $\alpha$ and $\beta$. Thus, by Definition~\ref{slope},
\begin{equation}\label{Def2}
 |\im F(w_{b_n}) - \im F(w_{a_n})|  \leq  \alpha \max\{\re F(w_{a_n}),\re F(w_{b_n})\} + \beta.
 \end{equation}
We deduce that, for large~$n$,
 \begin{equation}\label{modulus}
 \re F(w_{b_n}) > c |F(w_{b_n})|.
\end{equation}
Otherwise, by~(\ref{Lestimate}) and~(\ref{Def2}), we have, for large $n$,
\begin{eqnarray*}
|F(w_{b_n})| &\leq & \re F(w_{b_n})+|\im F(w_{b_n})| \\
& \leq & \re F(w_{b_n})+|\im F(w_{a_n})|+\alpha \max\{\re F(w_{a_n}),\re F(w_{b_n})\} + \beta\\
& \leq &(\alpha+2)\max\{|F(w_{a_n})|,\re F(w_{b_n})\} + \beta\\
& \leq & 2(\alpha+2)c|F(w_{b_n})|=\frac12 |F(w_{b_n})|,
\end{eqnarray*}
which is a contradiction.

Now let $D$ be the constant from Lemma~\ref{gulfs}. Since $w_1 \in I(F)$, for $n$ sufficiently large we have $a_n/D >  \max\{\re p_n,1\}$ and so, by Lemma~\ref{gulfs}, there exists $w_{a_n}' \in L_{p_n,a_n}$ such that
\[
  \re F(w_{a_n}') \geq  M(a_n/D,F).
\]

We now apply Lemma~\ref{Ahlfors} again but this time we replace $w_{a_n}$ by the point $w_{a_n}'$.
Since $a_n \to \infty$ and hence $|F(w_{a_n}')| \to \infty$ as $n \to \infty$, it follows from Lemma~\ref{Ahlfors} that, if $n$ is sufficiently large, then
\[
\frac{|F(w_{b_n})|}{ M(a_n/D,F)} \geq \frac{|F(w_{b_n})|}{|F(w_{a_n}')|} \geq \exp(\tfrac{1}{2}(b_n-a_n) - 4\pi).
\]
Together with \eqref{modulus} this implies that, if $n$ is sufficiently large, then
\begin{equation}\label{max}
\frac{\re F(w_{b_n})}{ M(a_n/D,F)} > c \exp(\tfrac{1}{2}(b_n-a_n) - 4\pi) > \exp(\tfrac{1}{4}(b_n-a_n)).
\end{equation}

 Note that, for large $n$, the part of $F^n(\Gamma)$ that connects $F^n(w_0)$ to $F^n(w_1)$ must intersect $L_{p_n,b_n}$  -- otherwise $F^n(w_1)$ lies in the bounded component of $T \setminus L_{p,b_n}$ and so the geodesic joining $F^n(w_1)$ to $\infty$ in the tract $T$ must contain a point $w$ with $\re w = b_n = \frac{1}{2K'} \re F^n(w_1)$. This, however, is
impossible for large $n$ since $w_1 \in I(F)$ and $T$ has bounded wiggling with constants $K'$ and $\mu$.

Thus, for large $n$, we can choose $w_{b_n} \in L_{p_n,b_n} \cap F^n(\Gamma)$ such that $F^n(w_1)$ is in the tail of
$F^n(\Gamma)$ from $w_{b_n}$ to $\infty$. So, if $K>1$ and $M>0$ are chosen so that $F$ satisfies a linear head-start condition for $K$ and $M$, then, for large $n$,
\[
\re F (w_{b_n}) \leq K \re F^{n+1}(w_1) + M
\]
and hence
\begin{equation}\label{5}
\re  F^{n+1}(w_1) \geq \frac{1}{K}\re F (w_{b_n}) - \frac{M}{K} \geq \frac{1}{2K} \re F(w_{b_n}).
\end{equation}

It follows from~\eqref{max} and~\eqref{5} that, for $n$ sufficiently large,
\begin{eqnarray*}
 \re F^{n+1}(w_1) & > & \frac{1}{2K}\exp(\tfrac{1}{4}(b_n- a_n)) M(a_n/D,F)\\
  & = & \frac{1}{2K}\exp(\tfrac{1}{4}(b_n- a_n))M(\tfrac{1}{4DK'}\re F^n(w_1),F).
\end{eqnarray*}
The result now follows on taking $\eps = \frac{1}{4DK'}$, since
\[
b_n-a_n = \frac{1}{4K'}\re F^n(w_1) \to \infty \; \mbox{ as }n \to \infty.
\]
\end{proof}

  Finally, we show that the conclusion of Theorem~\ref{infinite} follows easily from Lemma~\ref{A(f)}. We first choose $N_1 \in \N$ and $\eps \in (0,1)$ to satisfy Lemma~\ref{A(f)}. We then choose $r_1> 0$ and $N_2 \geq N_1$ such that
\begin{equation}\label{bound}
M(r,F) > r \mbox{ for } r \geq r_1  \quad\mbox{and}\quad \re F^{N_2}(w_1) >\frac{1}{\eps} r_1.
\end{equation}
Arguing by induction, it follows from~\eqref{bound} and Lemma~\ref{A(f)} that
\[
 \re F^{n+N_2}(w_1) > \frac{1}{\eps}M^n(r_1,F) > M^n(r_1,F), \mbox{ for all } n \in \N,
\]
and so $w_1 \in A(F)$ as required.
\end{proof}

\section{Slow Devaney hairs}
\setcounter{equation}{0}

 In this section we define a function $F\in\Blogn$ such that
  $F$ has a Devaney hair that
  is disjoint from $A(F)$. All the tracts of $F$ are translates of each other by $2n\pi i$, for some $n \in \Z$, and have bounded slope and bounded wiggling. This shows that an additional assumption, such as bounded gulfs, is indeed necessary
  in Theorem~\ref{infinite}.  We also construct a transcendental entire function which has a logarithmic transform with these properties.

  The function $F$ is $2\pi i$-periodic and the domain of $F$ consists of a tract
  $T\subset \{z\in \H: \im z\in (-\pi,\pi)\}$ and all
  $2\pi i \Z$-translates of $T$.
  Thus it is sufficient to specify the domain $T$ and a
  conformal isomorphism $F:T\to\H$.

  The tract $T$ is determined by an increasing sequence $(r_k)_{k\geq 0}$ of
  real numbers with $r_k>r_{k-1}+1$ and $r_0>1$, and a second sequence
  $(\eps_k)_{k\geq 0}$ of positive real numbers, as illustrated in Figure \ref{fig:slowhairs}.


  \begin{figure}
 \begin{center}
\resizebox{.98\textwidth}{!} {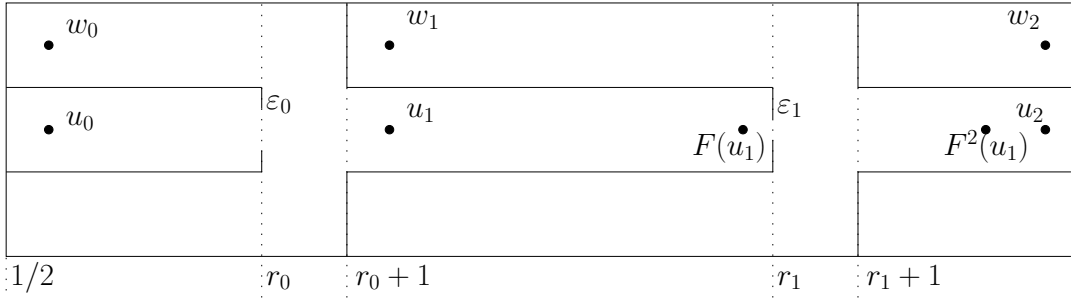}
 \end{center}
 \caption{\label{fig:slowhairs}Construction of the tract $T$.}
\end{figure}

  To be precise, $T$ consists of all points $u+iv$, where
  $u>1/2$ and $|v|<\pi$, subject to the following restrictions for
  all $k\geq 0$. (We also use the convention that $r_{-1}=-\infty$.)
 \begin{enumerate}
  \item If $r_{k-1} + 1 < u < r_k$, then $|v|\neq\pi/3$.
  \item If $u= r_k$, then $|v|>\pi/3$ or $|v|<\eps_k$.
  \item If $r_k < u < r_k + 1$, then there are no restrictions on $v$.
  \item If $u = r_k + 1$, then $|v|<\pi/3$.
 \end{enumerate}
{\it Remark.}\quad  The boundary of $T$ is not a Jordan curve, and hence
  formally speaking the $2\pi i$-periodic extension of a conformal
  isomorphism $F:T\to\H$ to $\T=T+2\pi i\Z$ does not belong to
  $\Blogn$. However, it will be clear from the construction that this defect is easily remedied by restricting $F$ to the
  set $F^{-1}(\{w:\re w > \eps\})$, for some $0<\eps<1/2$.

 Note that $T$ is symmetric about the real axis and has bounded slope and bounded wiggling. We choose
  $F:T\to\H$ to be the unique conformal isomorphism
  with $F(1)=1$ and $F(\infty)=\infty$.
  Note that $F$ maps the interval $[1,\infty)$ to itself, that~$1$ is the unique real fixed point of $F$, and that this fixed point
  is repelling. (To see this, apply the Schwarz lemma to the inverse branch
  $F^{-1}:\H\to T$.) It follows that $F(u)>u$ and $u\in I(F)$ for all $u>1$;
  in particular, $[1,\infty)$ is a Devaney hair of~$F$.

 In addition to the sequences $(r_k)_{k\geq 0}$ and $(\eps_k)_{k\geq 0}$, the definition of $T$ also involves sequences $(u_k)_{k\geq 0}$ and $(w_k)_{k\geq 0}$, where $u_k\in\R$ and
  $w_k = u_k + 2\pi i /3$. It is not difficult to show that it is possible to choose
  these sequences in such a way that, for each $k \geq 0$,
 \begin{enumerate}
 \item[(a)] $u_k,\,  F(u_k) \in (r_{k-1}+2,r_k - 1)$;
 \item[(b)] $u_{k+1} > F^2(u_k)$;
 \item[(c)] $\re F(w_k) > u_{k+1}$.
 \end{enumerate}
 (We give the details in Proposition \ref{prop:sequences} below.)
  These properties imply that  $(1,\infty) \subset I(F) \setminus A(F)$, as we now show.

\begin{lemma} \label{lem:slowhairs}
 Let $F:\T\to H$ belong to the class $\Blogn$.
   Suppose that a tract $T$ of~$F$ contains an interval
   $[a_0,\infty)$ of the real axis, and that $F$ is real; that is,
   $\T$ is symmetric with respect to the real axis and
   $F(\overline{z})=\overline{F(z)}$.

  Assume that there are sequences $u_k\in\R$ and $w_k\in\T$, with
  $u_k\to\infty$ and $u_k=\re w_k$, that satisfy properties (b) and (c) above. Then
   $\R\cap A(F)=\emptyset$. Further, there exists $a \in \R$ such that $[a, \infty)$ is a Devaney hair. Thus $F$ has a Devaney hair that is disjoint from $A(F)$.
\end{lemma}
 \begin{proof}

We first show that $\R \cap A(F) = \emptyset$.
 Since, for $k \geq 0$, we have
 \[ \re F(w_k) > u_{k+1} = \re w_{k+1}, \]
 we must have
 \begin{equation}\label{M^n}
 M^n(u_0,F) > u_{n}, \mbox{ for } n \in \N.
\end{equation}
 On the other hand, it follows from~(b) that, for $k, j\ge 0$,
\begin{equation}\label{2j}
 F^{2j}(u_k) < u_{k+j}.
\end{equation}
Now, let $u\in\R$ and let $k$ denote the smallest integer for which
  $u_k> u$. Let $L\in\N$ and set $n=L+2k$. Then, by~\eqref{M^n} and \eqref{2j},
\[ F^{n+L}(u) = F^{2L+2k}(u) < F^{2L+2k}(u_k) <
    u_{L+2k} = u_n < M^n(u_0,F). \]
 It follows that $u$ does not belong to $A(F)$, as claimed.

 We conclude the proof by noting that it follows from Lemma~\ref{Ahlfors} that there exists $a>0$ such that $[a, \infty) \in I(F)$. Further, $F^n(w) \to \infty$ uniformly on $[a, \infty)$ and  $F^n([a,\infty))$ is a simple curve connecting $F^n(a)$ to $\infty$. Hence $[a, \infty)$ is a Devaney hair that does not intersect $A(F)$.
\end{proof}

We now use Arakelian's theorem to prove that there exists a transcendental entire function such that a logarithmic transform of this function satisfies the hypotheses of Lemma~\ref{lem:slowhairs}. It seems likely that
 such a function can also be constructed to lie in the class $\B$,
 by using a similar argument to that in \cite[Section 7]{RRRS}, but we will not show
 this here.

\begin{theorem}\label{hair}
 There exists a transcendental entire function $h$ with a logarithmic tract $V$ such that
  $V$ contains a Devaney hair of~$h$ which does not meet $A'(h,V)$.
\end{theorem}
\begin{proof}
 Let the tract $T$ and the function $F:T \to \H$  be as described before Lemma~\ref{lem:slowhairs}, where the sequences
  $(r_k)$, $(\eps_k)$, $(u_k)$ and $(w_k)$ are chosen in such a way that
 \begin{enumerate}
 \item[(a${}'$)] $u_k,\,  F(u_k) \in (r_{k-1}+2,r_k - 1-k)$;
 \item[(b${}'$)] $u_{k+1} > F(r_k-1)+k$;
 \item[(c${}'$)] $\re F(w_k) > u_{k+1}+k$.
 \end{enumerate}
 (Again, we show that this is possible in
  Proposition \ref{prop:sequences}.)

 Let $\tilde{f}:\exp T\to \{z:|z|>1\}$ be defined by
   $\tilde{f}(\exp w) = \exp F(w)$, and let $f$ be the restriction of
   $\tilde{f}$ to
   \[A = \tilde{f}^{-1}(\{z:|z|\geq5/4\})=\exp F^{-1}(\{w:\re w\ge\ln(5/4)\}).\]

 Now, the complement of $A$ is connected and locally connected at $\infty$. Thus, by
 Arakelian's theorem \cite[page~142]{gaier}, there is an entire function $h:\C\to\C$
 such that $|f(z)-h(z)|<1/4$, for all $z\in A$. Since $f$ is real, we may
  also assume that $h$ is real (otherwise, consider the map
  $(h(z)+\overline{h(\overline{z})})/2$ instead).

 Let $V$ be the component of $h^{-1}(\{z:|z|>3/2\})$ that is contained in
  $A$ and contains some infinite piece of the positive real axis.
  Since $A$ contains no zeros of $h$, the set
  $V$ is simply connected by the minimum
  principle. Let $\tilde{T}$ be the component of $\exp^{-1}(V)$
   contained in $T$. There is a holomorphic map
   \[ H:\tilde{T}\to \{w: \re w > \ln(3/2)\} \]
  with $\exp\circ H = h\circ \exp$. By our choice of $h$, we have
    \begin{equation}\label{eqn:H_and_F}
     |H(w) - F(w)| < 1/4,  \quad\text{for all }w\in\tilde{T}.
     \end{equation}

  Note that $H$ is a proper map: indeed, $\re H(w)\to \ln(3/2)$ as
  $w\to \partial \tilde{T}$, while (\ref{eqn:H_and_F}) implies that
  $H(w)\to\infty$ as $w\to\infty$. So $H$ has a well-defined degree, and
  it is easy to see that this degree is one. Indeed, if $\gamma$ is
  a simple closed loop in $\tilde{T}$, sufficiently close to the boundary,
  then it follows from (\ref{eqn:H_and_F}) that $H(\gamma)$ winds exactly
  once around the point~$1$.

 So $H$ is a conformal isomorphism, and in particular $V$ is a logarithmic
  tract of $h$. We denote
  the $2\pi i$-periodic extension of $H$ also by $H$. Then
  $H\in\Blog$ is a logarithmic transform of $h$, which is not normalised.
  By (\ref{eqn:H_and_F}), and properties~(a${}'$), (b${}'$) and~(c${}'$), the map $H$ satisfies properties~(b)
   and~(c), for sufficiently large $k$. Thus, by Lemma~\ref{lem:slowhairs}, $H$ has a Devaney hair $[a,\infty)\subset\R$ that
  does not intersect $A(H)$. It follows
  that $[e^a,\infty)$ is a Devaney hair of $h$ that does not intersect
  $A'(h,V)$.
\end{proof}

{\it Remark.} It is in fact possible to construct a transcendental entire function $g$ with a logarithmic tract $V$ such that $V$ contains a component $C$ of $J(g)$ that is a Devaney hair of $g$ with $C \cap A'(g,V) \neq \emptyset$. In order to do this, we define $g(z) = h(z)/ \lambda$ and consider a tract $V$ of $h$, where $h$ and $V$ are as in the proof of Theorem~\ref{hair} and $\lambda > \max \{3/2, M(5/4,h)\}$. This ensures that $\{z: |z| \leq 5/4\} \subset F(g)$ and also that $V$ is a tract of $g$ with $\partial V \subset F(g)$. We then set $G(w) = H(w) - L$, where $L = \ln \lambda$ and denote the $2\pi i$-periodic extension of $G$ by $G$. The function $G$ is of {\it disjoint type}; that is, the tracts of $G$ are contained in the image of $G$ and so, arguing as in the proof of~\cite[Theorem 5.10]{RRRS}, it can be shown that there exists $a \in \R$ such that $[a, \infty)$ is a component of $J(G)$. Further, by Lemma~\ref{lem:slowhairs}, $[a,\infty)$ is a Devaney hair and $[a, \infty) \cap A(G) = \emptyset$. Then, because $G$ is of disjoint type and $\partial V \subset F(g)$, it can be shown that $C = [e^a, \infty)$ is a component of $J(g)$ that is a Devaney hair of $g$ with $C \cap A'(g,V) = \emptyset$.

To conclude the section, we show that the tract $T$ can
 indeed be chosen with the desired properties.
\begin{proposition}\label{prop:sequences}
 Let $(\delta_k)_{k \geq 0},(\eta_k)_{k \geq 0}$ and $(\theta_k)_{k \geq 0}$ be arbitrary sequences of positive real numbers.
 Then there exist sequences $(r_k)_{k \geq 0}, (\eps_k)_{k \geq 0}$ and $(u_k)_{k \geq 0}$ with $u_k \in T$, for $k \geq 0$, such that,
  if $T$ and $F$ are as described before Lemma~\ref{lem:slowhairs}, and $w_k = u_k + 2\pi i/3$, then
   \begin{itemize}
 \item[(A)] $u_k,\,  F(u_k) \in (r_{k-1}+2,r_k - 1 - \delta_k)$;
 \item[(B)] $u_{k+1} > F(r_k-1)+\eta_k$;
 \item[(C)] $\re F(w_k) > u_{k+1}+\theta_k$.
   \end{itemize}
\end{proposition}
\begin{proof}
 For a hyperbolic domain $U$, let $\dist_U(.,.)$ denote hyperbolic distance in~$U$ and $\rho_U(.)$ denote hyperbolic density in~$U$; see~\cite{CG}, for example. Since the curve $[1,\infty)$ is a hyperbolic geodesic in $T$, we have
    \[ \log(F(u)) = \dist_{\H}(1,F(u)) =
       \dist_{T}(1,u) = \int_{1}^u \rho_T(t)dt, \]
  for all $u\in [1,\infty)$. Now, by~\cite[Theorem~4.3]{CG}, we have
  \[
    \frac{1}{2\dist(w,\partial T)} \leq \rho_T(w) \leq \frac{2}{\dist(w,\partial T)}, \; \mbox{ for } w \in T.
  \]
   Therefore
   we can easily estimate the behaviour of $F$ in an initial segment of the tract, independently of the choices of
   $r_k$ and $\eps_k$ for large $k$.

  More precisely, these estimates for $\rho_T(.)$ show that if $u\in [1,\infty)$ and
  $u \geq r_{k}$, for $k\ge -1$, then there exist finite positive numbers $b(u)$ and $B(u)$ depending on $r_j$ and $\eps_j$, for $0 \leq j \le k$, such that
   \[
    b(u) \leq F(u) \leq B(u),
   \]
for all possible choices of $r_j$ and $\eps_j$, $j > k$, for which
 $r_{k+1} > u + 1$. Moreover $b(u) \to \infty$ as $\eps_k \to 0$.
 Note that $B(u)\geq F(u)\geq u$ for all $u\ge 1$.

 Also note that, for $k \geq 0$, we can connect $w_k$ to $r_{k}+1$ in $T$ by a curve
  of length
  \[(r_{k}+1)-u_k+2\pi/3 < 3r_{k}\]
  that has distance at least
  $1/2$ from $\partial T$. Hence the hyperbolic distance from $w_k$ to
  $r_{k}+1$ is at most $12r_{k}$. We can use this observation to bound the
  real part of $F(w_k)$ from below in terms of $F(r_k+1)$. Indeed,
  \[ \dist_T(w_k,r_k+1) = \dist_{\H}(F(w_k),F(r_k+1)) \geq
      \log(F(r_k+1)/\re F(w_k)), \]
   and hence
    \[ \re F(w_k) \geq F(r_k+1) \exp(-12 r_k). \]

 We now define the sequences $u_k$, $r_k$ and $\eps_k$ inductively.
  We begin by choosing
  $u_0=1$ and $r_0>\delta_0+2$.

 If we have defined $r_j$ and $u_j$ for $0 \leq j \leq k$,
  and $\eps_j$ for $0 \leq j < k$, then we choose (in order),
  $u_{k+1}$ so that
 \[  u_{k+1} > \max\{r_{k}+2,B(r_{k}-1)+\eta_k\}, \]
 $\eps_{k}$ so that
\[ \quad b(r_{k}+1)>(u_{k+1}+\theta_k)\exp(12r_{k})\]
  and $r_{k+1}$ so that
\[r_{k+1}> B(u_{k+1})+1+ \delta_{k+1}. \]

 (Recall that if $u\ge r_k$, then $b(u)\to\infty$ as $\eps_{k}\to 0$, and hence
   the desired choice of $\eps_k$ is indeed possible.)

 This completes the inductive description. Clearly properties (A) and (B) are
  satisfied by construction. To verify property (C), note that
   \begin{align*}
     \re F(w_k) &\geq F(r_k+1)\exp(-12r_k) \\ & \geq
        b(r_k+1)\exp(-12r_k) > (u_{k+1}+\theta_k), \end{align*}
   as desired.
\end{proof}
{\it Remark.}\quad
 Note that the proof shows that
  the sequences $\delta_k,\eta_k$ and $\theta_k$ could also be allowed to
  depend on previous values of $r_j$, $\eps_j$ and $u_j$. More precisely,
  they can depend on the values $u_j$, $j\leq k$, and on $r_j$ and
  $\eps_j$, $j<k$. Additionally, $\eta_k$ may also depend on $r_k$,
   and $\theta_k$ on $r_k$ and $u_{k+1}$.


\begin{thebibliography}{99}



\bibitem{kB07} K. Bara\' nski, Trees and hairs for some hyperbolic entire maps of finite order, {\it Math. Z.}, 257 (2007), no. 1, 33--59.

\bibitem{BK07} K.\ Bara\'nski and B. Karpi\'nska. Coding trees and boundaries of attracting basins for
some entire maps, {\it Nonlinearity}, 20 (2007) 391–-415.

\bibitem{wB93} W. Bergweiler, Iteration of meromorphic functions, {\it Bull. Amer. Math. Soc.}, 29 (1993), 151--188.

\bibitem{wB95} W. Bergweiler, On the Julia set of analytic self-maps of the punctured plane,
{\it Analysis}, 15 (1995), 251--256.

\bibitem{BH99} W. Bergweiler and A. Hinkkanen, On semiconjugation of entire functions,
{\it Math. Proc. Camb. Phil. Soc.}, 126 (1999), 565--574.

\bibitem{BRS08} W. Bergweiler, P.J. Rippon and G.M. Stallard, Dynamics of meromorphic functions with direct or logarithmic singularities,
{\it Proc. London Math. Soc.}, doi:10.1112/plms/pdn007.

\bibitem{CG} L. Carleson and T.W. Gamelin, {\it Complex Dynamics}, Springer, 1993.

\bibitem{DT86} R.L. Devaney and F. Tangerman, Dynamics of entire functions near the essential singularity, {\it Ergodic Theory Dynam. Systems}, 6 (1986), 498--503.

\bibitem{E} A.E. Eremenko, On the iteration of entire functions, {\it Dynamical systems and ergodic theory,} Banach Center Publications 23, Polish Scientific Publishers, Warsaw, 1989, 339--345.

\bibitem{EL} A.\ E.\ Eremenko and M.\ Yu.\ Lyubich,
Dynamical properties of some classes of entire functions, {\it Ann.\ Inst.\ Fourier (Grenoble)}, 42 (1992), 989--1020.

\bibitem{pF26} P. Fatou, Sur l'it\'{e}ration des fonctions transcendantes enti\`{e}res,
{\it Acta Math.}, 47 (1926), 337--360.

\bibitem{gaier} D. Gaier, {\it Lectures on Complex Approximation}, Birkh\" auser, 1987.

\bibitem{bK99} B. Karpi\'nska, Hausdorff dimension of the hairs without
endpoints for $\lambda \exp(z)$, {\it C. R. Acad. Sci. Paris S\'er. I Math.}, 328 (1999), 1039--1044.

\bibitem{Nev} R. Nevanlinna, {\it Analytic Functions}, Springer, 1970.

\bibitem{lR06} L.~Rempe, Topological dynamics of exponential maps on their escaping sets,
{\it Ergodic Theory Dynam. Systems}, 26 (2006), 1939--1975.

\bibitem{lR07}
L.\ Rempe, On a question of Eremenko concerning escaping components of entire
functions, {\it Bull.\ London Math.\ Soc.}, 39 (2007), 661--666.

\bibitem{lR08} L.\ Rempe, Rigidity of escaping dynamics for transcendental entire functions, to appear in {\it Acta Math.} arXiv: math/0605058.

\bibitem{RS05} P.J. Rippon and G.M. Stallard, On questions of Fatou and Eremenko,
{\it Proc. Amer. Math. Soc.}, 133 (2005), 1119--1126.

\bibitem{RRRS} G.\ Rottenfu{\ss}er, J. R\"uckert, L. Rempe and D.\ Schleicher,
Dynamic rays of bounded-type entire functions, to appear in {\it Ann. Math.}, arXiv: 0704.3213.

\bibitem{SZ} D. Schleicher and J. Zimmer. Escaping points of
exponential maps, {\it J. London Math. Soc.} (2), 67 (2003), 380--400.

\bibitem{Sta91} G.\ M.\  Stallard,
The Hausdorff dimension of Julia sets of entire functions.
{\it Ergodic Theory Dynam. Systems},  11  (1991), 769--777.

\bibitem{Sta00} G.\ M.\  Stallard, The Hausdorff dimension of Julia sets of entire functions IV,
 {\it J. London Math. Soc.} (2), 61 (2000), 471--488.

\end{thebibliography}
\end{document}